\documentclass[11pt,thmsa]{article}%
\usepackage{amsmath}
\usepackage{amsfonts}
\usepackage{amssymb}
\usepackage{graphicx}%
\newtheorem{theorem}{Theorem}[section]

\newtheorem{corollary}[theorem]{Corollary}

\newtheorem{definition}[theorem]{Definition}

\newtheorem{lemma}[theorem]{Lemma}

\newtheorem{proposition}[theorem]{Proposition}

\newenvironment{proof}[1][Proof]{\noindent\textbf{#1.} }{\ \rule{0.5em}{0.5em}}

\begin{document}
\date{}

\title{Clifford-Wolf homogeneous left invariant $(\alpha,\beta)$-metrics on
compact semi-simple Lie groups
\thanks{Supported by NSFC (no. 11271216, 11271198, 11221091) and SRFDP of China}}
\author{Ming Xu$^1$ and Shaoqiang Deng$^2$ \thanks{Corresponding author. E-mail: dengsq@nankai.edu.cn}\\
\\
$^1$College of Mathematics\\
Tianjin Normal University\\
 Tianjin 300387, P.R. China\\
 \\
$^2$School of Mathematical Sciences and LPMC\\
Nankai University\\
Tianjin 300071, P.R. China}
\maketitle

\begin{abstract}
Let $(M,F)$ be a connected Finsler space. An isometry of $(M,F)$ is
called a Clifford-Wolf translation (or simply CW-translation) if
it moves all points the same distance. The compact Finsler space $(M,F)$ is called
restrictively Clifford-Wolf homogeneous (restrictively CW-homogeneous) if for any two
sufficiently close points $x_1,x_2\in M$, there exists a CW-translation $\sigma$ such that
$\sigma(x_1)=x_2$. In this paper, we define the good normalized datum
for a homogeneous non-Riemannian $(\alpha,\beta)$-space, and use it
to study  the restrictive CW-homogeneity of left invariant  $(\alpha,\beta)$-metrics on a compact connected
semisimple Lie group. We prove that a  left invariant restrictively CW-homogeneous $(\alpha,\beta)$-metric on a compact semisimple Lie group  must be of the Randers
type. This gives a complete classification of left invariant
$(\alpha,\beta)$-metrics on compact semi-simple Lie groups which are
restrictively Clifford-Wolf homogeneous.



\end{abstract}

\section{Introduction}
The goal of this paper is to study left invariant restrictively Clifford-Wolf homogeneous (restrictively CW-homogeneous)
 $(\alpha,\beta)$-metrics on compact connected semi-simple Lie groups. Recall that an
isometry $\sigma$ of a metric space $(X,d)$ is called a Clifford-Wolf translation
(CW-translation) if the function $d(x,\sigma(x))$, $x\in X$, is a constant. A metric space is called CW-homogeneous if given any two points
$x_1,x_2\in M$, there is a CW-translation $\sigma$ such that $\sigma(x_1)=x_2$;
see \cite{BP99}.
There is a slightly weaker version of CW-homogeneity, called restrictive CW-homogeneity,
which only requires the existence of the CW-translation $\sigma$ with $\sigma(x_1)=x_2$
for sufficiently close $x_1$ and $x_2$; see Definition 2.2 below for the precise statement.
 Although a Finsler metric is not reversible in general, the
above definitions can be adapted to Finsler spaces by a word by word restatement.

The study of CW-translations has important merits in the investigations of the space
forms in Riemannian geometry; see Wolf's book \cite{WO10} for an excellent
survey. The related results have motivated a lot of mathematical activities; see
for example \cite{WO62,FR63,WO64,OZ69,OZ74,DMW86} for the determination of
CW-translations of explicit Riemannian manifolds; see also \cite{HE74,AW76} for the
applications of these results to the study of homogeneous Riemannian
manifolds of negative (non-positive) curvatures.

Recently, Berestovskii and Nikonorov studied the local one-parameter groups of
CW-translations of general Riemannian manifolds and established a
correspondence between local one-parameter groups of CW-translations and
Killing vector fields of constant length (KVFCLs); see
\cite{BN08-1,BN08-2,BN09}. The above research leads to a classification
of connected simply connected CW-homogeneous Riemannian manifolds.
The list consists of the Euclidean spaces, odd-dimensional spheres with
constant curvature, compact connected simply-connected Lie groups with
bi-invariant Riemannian metrics and Riemannian products of the above manifolds.
Notice that for simply-connected Riemannian manifolds, CW-homogeneity is
equivalent to restrictive CW-homogeneity.

More recently, we initiated the study of CW-translations of Finsler
spaces; see \cite{DX1,DX2}. The relation between local one-parameter group
of CW-translations and KVFCLs was generalized to the Finslerian case.
We classified CW-homogeneous left invariant
Randers metrics on compact simple Lie groups \cite{DX3} and all CW-homogeneous Randers metrics
on simply-connected manifolds \cite{XD1}. In this paper,
we will discuss a more generalized class of Finsler metrics, $(\alpha,\beta)$-metrics.
The main theorem is the following
\begin{theorem} \label{main-theorem}
Suppose $F=\alpha\phi(\beta/\alpha)$ is a left invariant restrictively CW-homogeneous
$(\alpha,\beta)$-metrics on a compact connected simple Lie group $G$, then
$F$ must be a Randers metric.
\end{theorem}

Combined with the classification theorem in \cite{XD1} (which is still correct with
CW-homogeneity changed to restrictive CW-homogeneity),  this theorem provides
a complete classification of restrictively CW-homogeneous left invariant
$(\alpha,\beta)$-metrics
on compact semi-simple Lie groups.
Using some similar arguments as in \cite{DX3} or \cite{XD1},
%where only requires the
%bi-invariant Riemannian metric and left invariant vector fields for navigation, rather than that $G$ is simple or simply-connected,
we can prove that a left invariant restrictively CW-homogeneous Finsler metric
on a compact semisimple Lie group is actually CW-homogeneous. Therefore,  Theorem \ref{main-theorem}
 also gives a complete classification of CW-homogeneous left invariant
$(\alpha,\beta)$-metrics on compact semisimple Lie groups.

Theorem \ref{main-theorem} is not  valid for a general
compact Lie group. For example, let $G=G'\times S^1$,
 where $G'$ is a compact semi-simple Lie group, $\alpha$ is a bi-invariant metric on $G$, and $\beta$ is a $\alpha$-parallel 1-form induced by the standard 1-form
on the $S^1$-factor, then for any smooth function satisfies the condition below, the  $(\alpha,\beta)$-metric $F=\alpha\phi(\beta/\alpha)$
is CW-homogeneous. Note that such a metric must be a Berwald metric. 
%; see \cite{DH}. ( This reference is not founded. [DH02] is not relavent to the statement. ) 

A very interesting and difficult problem is to classify all
the CW-homogeneous Finsler spaces. It seems much more difficult than the same problem
for Randers spaces.
% ( I have known the conjectures not true. )

In Section 2, we present some known results on related topics. In Section 3, we
define  the good normalized datum for a homogeneous non-Riemannian
$(\alpha,\beta)$-space $(M,F)$, and gives some method to find such datum. In Section 4, we use the good normalized datum
to study the space of KVFCLs of a left invariant restrictively CW-homogeneous
$(\alpha,\beta)$-metric on a compact connected simple Lie group and prove
Theorem \ref{main-theorem} for compact connected simple Lie groups. Finally, in Section 5, we prove Theorem \ref{main-theorem} for all compact connected semi-simple
Lie groups by mathematical induction.

\section{Preliminaries}

\subsection{The definition and examples of Finsler metrics}
A Minkowski norm on a $n$-dimensional real linear space $ V$ is a continuous function
$F:   V \rightarrow [0,+\infty)$ satisfying the following conditions:
\begin{enumerate}
\item (Positivity)\quad $F(y)$ is a positive smooth function on 
$ V\backslash 0$.
\item (Positive homogeneity)\quad $F(\lambda y)=\lambda F(y)$ for any $\lambda>0$.
\item (Strong convexity)\quad The Hessian matrix
\begin{equation}
(g_{ij}(y))=(\frac{1}{2}[F^2(y)]_{y^iy^j})
\end{equation}
is positive definite on $  V\backslash 0$.
\end{enumerate}

 The Minkowski norm $F$ is called Euclidean or a linear metric  if its Hessian matrix is independent of $y$, i.e.,
if $F^2=g_{ij}y^i y^j$ is defined by an inner product on $ V$.

Let $M$ be a connected smooth manifold. A Finsler metric on $M$ is a continuous
function $F: TM\rightarrow [0,+\infty)$ which is smooth on the slit tangent bundle
$TM\backslash 0$, such that its restriction to each tangent space is a Minkowski norm.

The pair $(M,F)$ is called a Finsler space or a Finsler manifold.
It is a Riemannian manifold if its restriction in each tangent space is an
Euclidean norm (a linear metric).

The most important examples of non-Riemannian Finsler metrics are  Randers metrics.
A Randers metric is a  Finsler metric of the form $F=\alpha+\beta$, where $\alpha$ is a Riemannian metric
and $\beta$ is a 1-form whose $\alpha$-length is less than $1$ everywhere. Randers metrics were introduced
by G. Randers in 1941, in his study of general relativity \cite{RA41}.

There are a more generalized class of Finsler metrics which have been studied extensively
in the literature. Let $\alpha$ be a Riemannian metric and $\beta$ a smooth 1-form on
the manifold $M$. An $(\alpha,\beta)$-metric is a Finsler metric of the form $F=\alpha\phi(\beta/\alpha)$, where
$\phi$ is a positive function on $\mathbb{R}$. The condition for $F$ to define a Finsler metric on $M$ can be stated as follows (see \cite{CS05}). Denote
$\epsilon_0=\sup_{(x,y)\in TM\backslash 0}\beta(x,y)/\alpha(x,y)$. If $\epsilon_0$ can be
attained at certain point $(x_0,y_0)$ and it is positive, then  $\phi$ is required to be smooth on
$\mathcal{I}=[-\epsilon_0,\epsilon_0]$ and
satisfies
\begin{equation}\label{f-2-1}
 \phi(s)-s\phi'(s)+(b^2-s^2)\phi''(s)>0,
\end{equation}
for all $b$ and $s$ such that $|s|\leq |b|\leq \epsilon_0$; If $\epsilon_0$
can not be attained at any point, then $\phi$ is required to be  positive
and smooth  on $\mathcal{I}=(-\epsilon_0,\epsilon_0)$ ($\mathcal{I}=\mathbb{R}$ when
$\epsilon_0=\infty$), and (\ref{f-2-1}) is satisfied for all $b$ and $s$ with
$|s|\leq |b|<\epsilon_0$. Notice that the Riemannian metric  $\alpha$ or the $1$-form  $\beta$ in the definition of a
$(\alpha,\beta)$-metric may not be unique.
When $\beta$ is identically $0$, the metric $F$ is Riemannian. If $\phi$ is a linear
function, then $F$ is a Randers metric.

\subsection{Homogeneous Finsler spaces}
On a Finsler space $(M,F)$ one can define the arc length of a piecewise smooth path.
Let $x,x'\in M$. Then the  distance  $d(x,x')$ is defined to be the supremum of the arc lengths of all piecewise smooth paths from $x$ to $x'$. Notice that in general we do not have the reversibility $d(x,x')\equiv d(x',x)$, unless $F$ is reversible, i.e.,
$F(x,y)=F(x,-y)$ for any $x\in M$ and $y\in T_x M$. An isometry $\varphi$ of $(M,F)$
is  a diffeomorphism of $M$ such that $\varphi^* F=F$.  Equivalently, an isometry is  a homoemorphism of $M$ onto itself such that
 $d(x,x')=d(\varphi(x),\varphi(x'))$ for any $x,x'\in M$ (see \cite{DH02}). It was proven by Deng and Hou that the group $I(M,F)$ of all isometries of $(M,F)$,
endowed with the open-compact topology, is a Lie group \cite{DH02}. The space $(M,F)$ is called a homogeneous Finsler space if $I(M,F)$ acts transitively on $M$. In this case
 the manifold $M$ can be written as  a coset space $G/H$, where $G$ is a closed subgroup of $I(M,F)$ which
acts transitively on $M$ and $H$ is the isotropy subgroup of $G$ at a point $x_0\in M$. In
general, there may be more than one way to write $M$ as $G/H$.  Since in this paper we will only consider connected manifolds, the subgroup $G$ can also be chosen to be a closed
connected subgroup of the connected isometry group $I_0(M,F)$.

Let us give some examples of homogeneous Finsler spaces.

Let $G$ be a  a Lie group.  A Finsler metric $F$ on $G$ is called left invariant if $L(G)\subset
I(G,F)$. Then $(G, F)$  is obviously homogeneous.

The second example is a homogeneous Randers metric $F=\alpha+\beta$ on $M=G/H$.
The uniqueness of the presentation of $F$ indicates that both $\alpha$ and $\beta$
are preserved under the action of $I(M,F)$. By a 1-to-1
correspondence, the metric $F$ is determined by the restrictions of
$\alpha$ and $\beta$ in $T_{x_0}M=\mathfrak{m}=\mathfrak{g}/\mathfrak{h}$,
i.e. an $\mathrm{Ad}(H)$-invariant linear metric on $\mathfrak{m}$ and an
$\mathrm{Ad}(H)$-invariant vector in $\mathfrak{m}^*$.

Now we consider a homogeneous $(\alpha,\beta)$-metric $F=\alpha\phi(\beta/\alpha)$
on $M=G/H$. In general
there exists more than one way to write $F$  as an $(\alpha,\beta)$-metric, hence  $\alpha$ and $\beta$ may not be $G$-invariant, or equivalently,  their restrictions in
$\mathfrak{m}$ may not be $\mathrm{Ad}(H)$-invariant. To tackle this problem, we introduce the notion of a good datum. A triple $(\phi,\alpha,\beta)$ is called
a good datum of the homogeneous non-Riemannian $(\alpha,\beta)$-metric $F=\alpha\phi(\beta/\alpha)$ if both $\alpha$ and $\beta$ are invariant under the action of $I_0(M,F)$.
The following properties of a good datum are easy to verify:
\begin{enumerate}
\item For any closed connected subgroup $G\subset I_0(M,F)$ which acts transitively on
$M$, the restrictions of $\alpha$ and $\beta$ in $\mathfrak{m}$ are $\mathrm{Ad}(H)$
invariant, where $H$ is isotropy subgroup of $G$ at $x_0\in M$.
\item The isometry group of $(M,F)$ can be identified with  the closed subgroup
of $I(M,\alpha)$ which keeps $\beta$ invariant. A vector field $X$ is a Killing vector field of
$(M,F)$ if and only if $X$ is a  Killing vector field of $(M,\alpha)$ and
$L_X \beta=0$.
\end{enumerate}

We will show how to find a good datum for a homogeneous $(\alpha,\beta)$-metric
in the next section.

\subsection{CW-translations and CW-homogeneity of Finsler spaces}

Although the distance function of a Finsler space is generally not symmetric, the notions of CW-translations and CW-homogeneity of Finsler spaces can be
defined in the same way as for metric spaces. For the completeness of the article
we briefly recall the definitions below.

\begin{definition}
An isometry $\rho$ of a Finsler space $(M,F)$ is called a Clifford-Wolf translation (CW-translation)
if $d(x,\rho(x))$ is a constant function for $x\in M$.
\end{definition}

\begin{definition}
A Finsler space $(M,F)$ is called Clifford-Wolf homogeneous (CW-homogeneous)
if for any pair of points $x,x'\in M$, there is a CW-translation $\rho$ which
sends $x$ to $x'$. It is called restrictively CW-homogeneous, if for any $x$,
there is a neighborhood ${U}$ of $x$, such that for any pair of points $x_1$ and $x_2$ in ${U}$,
there is a CW-translation $\rho$ of $(M,F)$, such that $\rho(x_1)=x_2$.
\end{definition}

The main tool to study CW-translations and CW-homogeneity in Finsler geometry is
a natural interrelation between Killing vector fields of constant length
(KVFCLs) and local one-parameter semigroups of CW-translations. We
now recall the main results in \cite{DX1}.

\begin{theorem}
Suppose $(M,F)$ is a complete Finsler manifold with positive injective radius. If
$X$ is a KVFCL of $(M,F)$ and $\varphi_t$ is the flow generated by $X$, then
$\varphi_t$ is a Clifford-Wolf translation for any sufficiently small $t>0$.
\end{theorem}

\begin{theorem}\label{thm-2-4}
Let $(M,F)$ be a compact Finsler space. Then there is a $\delta>0$, such that
any CW-translation $\rho$ with $d(x,\rho(x))<\delta$ is contained in a local
one-parameter semigroup of CW-translations generated by a KVFCL.
\end{theorem}

Notice that Theorem \ref{thm-2-4} is still correct if we replace the compactness of $M$
by the homogeneity of $(M,F)$; see \cite{XD1}.

Based on these interrelation theorems, we have an equivalent description of the
restrictive CW-homogeneity.

\begin{proposition}\label{pro-2-5}
Let $(M,F)$ be a compact connected homogeneous Finsler space. Then it is restrictively
CW-homogeneous if and only if any tangent vector can be extended to a KVFCL of $(M,F)$.
\end{proposition}

\section{Good normalized data of a homogeneous non-Riemannian
$(\alpha,\beta)$-space}

\subsection{Non-Riemannian $(\alpha,\beta)$-norms,  linear isometry groups and normalized data}
Before discussing homogeneous $(\alpha,\beta)$-spaces, let us  look at its local model.

%An $(\alpha,\beta)$-norm on a linear space $ V$ is a Minkowski norm
%$F$ of the form $F=\alpha\phi(\beta/\alpha)$, in which $\alpha$ is linear metric
%on $ V$, and $\beta$ is a vector in $ V^*$. When $\epsilon=||\beta||_\alpha>0$,
%the function $\phi$ must be smooth on $[-\epsilon,\epsilon]$ and satisfy (\ref{f-2-1}) for
%any $s$ and $b$ with $|s|\leq |b|\leq \epsilon$.

Let $F$ be a Minkowski norm on a real linear space $ V$. Denote by $L( V,F)$ the group of linear isometries of $( V,F)$ (which is a compact Lie group),   and by $L_0( V,F)$
the unity component of $L( V, F)$. Then we have

\begin{lemma}\label{lem-3-1}
(1)\quad  Suppose  $\dim  V=n>1$. Then the Minkowski norm $F$  is an $(\alpha,\beta)$-norm if and only if there is a linear metric $\alpha$ and an $\alpha$-orthogonal decomposition
$ V=  V_1\oplus  V_2$, with $\dim  V_1=n-1$, such that $L_0(  V,F)$ contains
$\mathrm{SO}(  V_1,\alpha)$,
the maximal connected subgroup of linear isomorphisms
which preserve $\alpha$ and act trivially on $  V_2$.

(2)\quad  An $(\alpha,\beta)$-norm $F$ is Riemannian if and only if $\dim L_0(  V,F) >\dim \mathrm{SO}(  V_1,\alpha)$.
\end{lemma}

\begin{proof}(1)\quad
Suppose $F=\alpha\phi(\beta/\alpha)$ is an $(\alpha,\beta)$-norm. If $\beta=0$ we can
choose any $\alpha$-orthogonal
decomposition as indicated in the lemma. Then we have
\begin{equation}
\mathrm{SO}(  V_1,\alpha)\subset \mathrm{SO}(  V,\alpha)=L_0( V,F).
\end{equation}
If $\beta\neq 0$, we can take $ V_1=\mathrm{ker}\beta$ and $ V_2$ to
be the $\alpha$-orthogonal complement of $  V_1$. The functions $\alpha$ and $\beta$
take the same value on each $\mathrm{SO}(  V_1,\alpha)$-orbit, so does $F$. Thus
$L_0(  V,F)$ contains the connected subgroup $\mathrm{SO}(  V_1,\alpha)$.

Conversely, assume that we can find $\alpha$ and an $\alpha$-orthogonal decomposition
$  V=  V_1\oplus  V_2$ with $\dim  V_1=n-1$, such that
$\mathrm{SO}(  V_1,\alpha)\subset L_0(  V,F)$. Then we can choose a
nonzero $\beta\in  V^*$ such that $\mathrm{ker}\beta=  V_1$. If
$y_1,y_2\in  V$ and  $\alpha (y_1)=\beta (y_2)$, then $y_1$ and  $y_2$ must be contained   in
the same orbit of $\mathrm{SO}(  V_1,\alpha)$. Since  $\mathrm{SO}(  V_1,\alpha)\subset L_0(  V,F)$, we have $F(y_1)=F(y_2)$. Thus
$F$ only depends on the values of $\alpha$ and $\beta$. Hence we can find a suitable function $\phi$
such that $F=\alpha\phi(\beta/\alpha)$.

(2)\quad
 Up to conjugation, $\mathrm{SO}( V,\alpha)$ is just the standard special orthogonal subgroup $\mathrm{SO}(n)$ and
$\mathrm{SO}( V_1,\alpha)$ the standard subgroup $\mathrm{SO}(n-1)\subset \mathrm{SO}(n)$. We have seen in the
above argument that, if $F$ is Riemannian, then $L_0( V,F)$, which is isomorphic to $\mathrm{SO}(n)$,
has a larger dimension than $\mathrm{SO}( V_1,\alpha)$.

Conversely,
assume that $\dim L_0( V,F) > \dim \mathrm{SO}( V_1,\alpha)$.
Then we can find an infinitesimal
generator $X$ of $L_0( V,F)$, nonzero vectors $ V_1\in  V_1$ and $ V_2\in  V_2$, such that $X( V_1)= V_2$ and $X( V_2)=- V_1$. Now $X$ generates an one-parameter of isometries, which is  just the action of $S^1$  as
rotations on the $2$-dimensional subspace $W$ generated by $ V_1$ and $ V_2$. Then the restriction of $F$ to $\mathbf{W}$  is invariant under the rotations generated by $X$, hence $F|_\mathbf{W}$ it is a Euclidean norm. Therefore $F$ must be of the form $\sqrt{a\alpha^2+b\beta^2}$ for some constants $a$ and $b$, and it is
a linear metric on $ V$.
\end{proof}

According to Lemma \ref{lem-3-1}, when $\dim V>2$, a non-Riemannian
$(\alpha,\beta)$-norm $F=\alpha\phi(\beta/\alpha)$ on $ V$ determines
a unique decomposition of $ V$ into the direct sum of
irreducible representations of $L_0( V,F)$, i.e., $ V= V_1+ V_2$, such that $ V_1$ is $(n-1)$-dimensional with
the natural action of $L_0( V,F)$, and $ V_2$ is $1$-dimensional with the
trivial action of $L_0( V,F)$. Since $\mathrm{ker}\beta= V_1$, $\beta$
is uniquely determined by $F$ up to a scalar multiplication. Moreover, the linear metric $\alpha$ is also uniquely determined
by $F$ in the sense that there are two positive scalars $c_1,c_2$ such that $\alpha|_{V_1}=c_1F|_{V_1}$ and $\alpha|_{V_2}=c_2F|_{V_2}$.

A triple $(\phi,\alpha,\beta)$ is called a normalized datum if we have
\begin{eqnarray*}
\alpha|_{ V_1}=F|_{ V_1},
\end{eqnarray*}
and for any $y\in V_2$ with $\beta (y) >0$, we have
\begin{eqnarray*}
\alpha (y)=\beta (y)=F (y).
\end{eqnarray*}

For a normalized datum $(\phi,\alpha,\beta)$,
we have $||\beta||_\alpha=1$. Thus
 $\phi$ is a smooth function on $[-1,1]$, and $\phi(0)=\phi(1)=1$.

The following corollary is obvious.
\begin{corollary}
Let $F$ be a non-Riemannian $(\alpha,\beta)$-norm on a real
linear space $ V$,  with $\dim V>2$. Then there are at most two normalized
data of $F$. Moreover, in any normalized datum $(\phi,\alpha,\beta)$ of $F$, $\alpha$ and $\beta$ are
invariant under the action of $L_0( V,F)$.
\end{corollary}
\subsection{An existence theorem}
Now we turn  back to  homogeneous non-Riemannian $(\alpha,\beta)$-spaces. The following
theorem tells us  that in most cases a good datum exists.  Moreover, the proof of the following problem shows how to find a good normalized datum.

\begin{theorem}\label{thm-3-3}
Let $(M,F)$ be a homogeneous non-Riemannian Finsler space such that the restriction  of $F$ to any tangent space is an  $(\alpha,\beta)$-norm. Suppose there is  a closed connected subgroup $G$  of $I_0(M,F)$ which acts transitively on $M$  such that the isotropy subgroup $H$
at a point $x_0\in M$ is connected. Then $F$ is an $(\alpha,\beta)$-metric. Moreover,  there is  a good global datum $(\phi,\alpha,\beta)$ of $F$ such that the restriction of the datum
to any tangent space  is a normalized datum.
\end{theorem}

\begin{proof}
First we construct the global datum $(\phi,\alpha,\beta)$ for $F$. Notice that since  $F$ is  non-Riemannian,  the restriction of $F$ to  a tangent space cannot be a linear metric. In particular, the restriction of $F$ to   $T_{x_0}M$ is a non-euclidean $(\alpha,\beta)$-norm, hence  there exists a normalized
datum $(\phi,\alpha,\beta)$ for   $F(x_0,\cdot)$. Then for any $g\in G$,
$(\phi,g^*\alpha,g^*\beta)$ is a normalized datum for $F({g^{-1}x_0},\cdot)$. Now for any two elements $g$ and
$g'$ in $G$ such that $g^{-1}x_0=g'^{-1}x_0$, $gg'^{-1}\in H$ defines an element in
$L_0(TM_{x_0},F(x_0,\cdot))$ by the connectedness of $H$. So we have $gg'^{-1}\alpha=\alpha$ and $gg'^{-1}\beta=\beta$. Thus the normalized data $(\phi,g^*\alpha,g^*\beta)$ and $(\phi,g'^*\alpha,g'^*\beta)$
coincide. By the smoothness of the action, these data form a smooth global datum $(\phi,\alpha,\beta)$ for
$F=\alpha\phi(\beta/\alpha)$. Therefore  $F$ is an $(\alpha,\beta)$-metric.

Given $\rho\in I_0(M,F)$, there is a continuous family $\rho_t\in I_0(M,F)$ such that
 $\rho_0=\mathrm{id}$ and  $\rho_1=\rho$.  For each $x\in M$,
$(\phi,\rho_t^*(\alpha|_{\rho_t(x)}),\rho_t^*(\beta|_{\rho_t(x)}))$ is a continuous family
of normalized data for the $(\alpha,\beta)$-norm $F(x,\cdot)$. It must be a constant family. Thus  $\rho^*\alpha=\alpha$ and $\rho^*\beta=\beta$. Hence the normalzied datum $(\phi,\alpha,\beta)$ is  a good datum.
\end{proof}

As an  example, let $F$ be a left invariant non-Riemannian $(\alpha,\beta)$-metric on a compact connected
Lie group $G$. Denote $G'=I_0(G,F)$. Then   $L(G)\subset G'$, and the Lie group $G$ can be written as   $G=G'/H$, where
$H$ is the isotropy subgroup of $G'$ at $e\in G$. Since $G'$ is diffeomorphic to the product of
$G$ and $H$,  $H$ is connected. By Theorem \ref{thm-3-3}, we can find a good normalized datum for $F$.

%Even when the isotropy subgroup is not connected, it is not a serious issue. For a lot of
%geometric problems we can turn to a cover of $M$ with an induced metric by $F$ and connected isotropy subgroup for
%discussion. Then the good normalized datum can be found.

\section{Restrictive CW-homogeneity and left invariant $(\alpha,\beta)$-metrics on a compact connected simple $G$}
\subsection{Some notations}
We first introduce  some notations which will be used throughout this section.

Let $G$ be a compact connected simple Lie group, and
$F$  a left invariant non-Riemannian $(\alpha,\beta)$-metric on $G$.
Theorem \ref{thm-3-3} indicates that we can find a good normalized datum for $F$. We denote  the restriction of the datum to
$T_e G=\mathfrak{g}$  as $(\phi,\alpha,\beta)$.
Meanwhile, the $(\alpha,\beta)$-norm defined by $F$ in $\mathfrak{g}$ will also be
denoted as $F=\alpha\phi(\beta/\alpha)$.

By Theorem \ref{thm-3-3} and \cite{OT76}, we have
$I_0(G,F)\subset L(G)R(G)$. Let $G'$ be the maximal connected closed subgroup of $G$, such that
$R(G')$ consists of isometric right translations.
Then $I_0(G,F)=L(G)R(G')$.
Denote $\mathrm{Lie}(G)=\mathfrak{g}$
and $\mathrm{Lie}(G')=\mathfrak{g}'$. The space of Killing vector fields of $(G,F)$
can be identified with the Lie algebra of $I_0(G,F)$, i.e., $\mathfrak{g}\oplus\mathfrak{g}'$.
The isotropy subgroup at $e\in G$ is isomorphic to $G'$, whose Lie algebra is $\{(X', X')|X'\in \mathfrak{g}'\}$. The group $G'$ can also be identified with the maximal  connected closed subgroup
of $G$ whose $\mathrm{Ad}$-action preserves the functions $\alpha$ and $\beta$ on $\mathfrak{g}$.

The inner product defined by  $\alpha$ on $\mathfrak{g}$ will be denoted as $\langle\cdot,\cdot\rangle$, and the inner product corresponding to  the bi-invariant linear metric
$||\cdot||_{\mathrm{bi}}$ will be denoted as $\langle\cdot,\cdot,\rangle_{\mathrm{bi}}$.
Let $ v$ and $ v'$ be the  nonzero vectors in $\mathfrak{g}$ such that
$\beta(\cdot)=\langle  v,\cdot\rangle=\langle  v',\cdot\rangle_{\mathrm{bi}}$. Then it is
easy to see that $G'\subset C_G( v)$ and $\mathfrak{g}\subset \mathfrak{c}_\mathfrak{g}( v)$.

\subsection{The decomposition of the set of  KVFCLs}
The interrelation between CW-translations and KVFCLs, and in particular Proposition
\ref{pro-2-5} suggest that  we should study the set of KVFCLs.

We first prove  a similar criterion for a Killing vector field
$(X,X')\in\mathfrak{g}\oplus\mathfrak{g}'$ to have constant length as in the Randers case
\cite{DX3}
\begin{theorem}\label{thm-4-1}
If $(X,X')\in\mathfrak{g}\oplus\mathfrak{g}'$ defines a KVFCL of $F$, then either $X=0$ or
$X'\in\mathfrak{c}(\mathfrak{g}')$.
\end{theorem}

\begin{proof}
The Killing vector field defined by $(X,X')$ has the $F$-length
$F(\mathrm{Ad}(g)X-\mathrm{Ad}(g')X')$ at $g g'^{-1}$, for $g\in G$ and $g'\in G'$.
If $(X,X')$ defines a KVFCL, then
\begin{equation}\label{f-4-5}
\alpha(\mathrm{Ad}(g)X-\mathrm{Ad}(g')X')\phi(\frac{\beta(\mathrm{Ad}(g)X-\mathrm{Ad}(g')X')}{\alpha(\mathrm{Ad}(g)X-\mathrm{Ad}(g')X')})=\mathrm{const},\forall g\in G, g'\in G'.
\end{equation}
Thus for a fixed $g\in G$, $\beta(\mathrm{Ad}(g)X-\mathrm{Ad}(g')X')=\beta(Ad(g)X-X')$
is a constant function of $g'$. For  $Y\in\mathfrak{g}'$ and
$g'_0\in G'$, denote
\begin{equation}
X_{g'_0,t,Y}=\mathrm{Ad}(g)X-\mathrm{Ad}(\exp(tY)g'_0)X',
\end{equation}
$s_{g'_0,t,Y}=\beta(X_{g'_0,t,Y})/\alpha(X_{g'_0,t,Y})$, and
\begin{equation}
s_0=\frac{\beta(\mathrm{Ad}(g)X-\mathrm{Ad}(g'_0)X')}{\alpha(\mathrm{Ad}(g)X-\mathrm{Ad}(g'_0)X')}.
\end{equation}
Setting $g'=\exp(tY)g'_0$ in (\ref{f-4-5}),  taking the differential   with respect to $t$
and considering the value at $t=0$, we have
\begin{equation}
(\phi(s_0)-s_0\phi'(s_0))\frac{d}{dt}\alpha(X_{g'_0,t,Y})|_{t=0}=0,
\forall Y\in\mathfrak{g}', g_0\in G'.
\end{equation}
By (\ref{f-2-1}) and (\ref{f-4-5}), $\alpha(\mathrm{Ad}(g)X-\mathrm{Ad}(g')X')$ must
also be a constant function of $g'$. Note that neither $\alpha(\mathrm{Ad}(g)X)$ nor
$\alpha(\mathrm{Ad}(g')X')=\alpha(X')$ depends on $g'$, so
$\langle \mathrm{Ad}(g)X,\mathrm{Ad}(g')X'\rangle$
is a constant function of $g'$. Thus for any $g\in G$, $\mathrm{Ad}(g)X$ is
$\alpha$-orthogonal to the ideal generated by $[X',\mathfrak{g}']$ in $\mathfrak{g}'$.
Now change $g$ arbitrarily, we can prove that   the ideal of $\mathfrak{g}$ generated by
$[X,\mathfrak{g}]$ is
$\alpha$-orthogonal to the ideal of $\mathfrak{g}'$generated by $[X',\mathfrak{g}']$.
If $X\neq 0$, then $X'$ generates the $0$ ideal in $\mathfrak{g}'$. Thus
$X'\in\mathfrak{c}(\mathfrak{g}')$. This completes the proof.
\end{proof}

For simplicity, we denote the set of  KVFCLs of the metric $F$ as $\mathcal{K}_F$.
Theorem \ref{thm-4-1} implies that  $\mathcal{K}_F$ can be decomposed into the union of
$\mathcal{K}_{F;1}$ and $\mathcal{K}_{F;2}$, where $\mathcal{K}_{F;1}$ is the closure
of the set of  KVFCLs $(X,X')$ with $X\neq 0$, and $\mathcal{K}_{F;2}$ is the
linear subspace $0\oplus\mathfrak{g}'$. The following lemma shows that
$\mathcal{K}_{F;1}\cap\mathcal{K}_{F;2}=\{0\}$.

\begin{lemma}\label{lem-4-2}
There is a constant $C>0$, such that for any KVFCL $(X,X')\in\mathcal{K}_{F;1}$, we have
 $||X'||_{\mathrm{bi}}<C||X||_{\mathrm{bi}}$.
\end{lemma}

\begin{proof}
The Lie algebra $\mathfrak{g}$ will be viewed as a flat manifold with the metric $\langle\cdot,\cdot\rangle_{\mathrm{bi}}$, and any submanifold in it will be endowed
with the induced metric.

Suppose conversely that the constant $C>0$ indicated in the lemma does not exist. Then
there is a sequence of $(X_n,X'_n)\in\mathcal{K}_{F;1}$ such that
$||X_n||_{\mathrm{bi}}=1$, $X'_n\in\mathfrak{c}(\mathfrak{g}')$ with
$\mathop{\lim}\limits_{n\rightarrow\infty}||X'_n||_{\mathrm{bi}}=\infty$.
Denote $F(Ad(g)X_n-X'_n)=l_n$. Then the sequence $\{l_n\}$ also diverges to $\infty$.
The $\mathrm{Ad}(G)$-orbit $\mathcal{O}_{X_n}$ is contained in
the hypersurface
\begin{equation}
\mathcal{S}_n=\{Y|F(Y-X'_n)=l_n\}\subset \mathfrak{g},
\end{equation}
on which the $C^0$-norm of all
principal curvatures converges to $0$ when $n\rightarrow \infty$. Taking  a suitable sequence if necessary,
we can assume that  $\mathop{\lim}\limits_{n\rightarrow\infty}X_n=X$. Then in the closed round ball with center $0$
and radius $3$ (with respect to the bi-invariant metric), the hypersurfaces
$\mathcal{S}_n$ converges to a flat hyperplane $\mathcal{S}$ of codimension $1$
in $\mathfrak{g}$. Hence the hyperplane $\mathcal{S}$ contains the $\mathrm{Ad}(G)$-orbit
$\mathcal{O}_X$ of the nonzero vector $X$. This can not happen for a compact connected simple $G$.
\end{proof}

The KVFCLs in $\mathcal{K}_{F;2}$ or the CW-translations generated by them are
relevant to $I_0(G,F)$ rather than $F$ itself.
Therefore they are of little interest to our study. The following corollary shows that in most cases, we only need
to consider the KVFCLs in $\mathcal{K}_{F;1}$.

\begin{corollary}\label{cor-4-3}
Keep all the notations as above. The metric $F$ is restrictively CW-homogeneous if and only if any
nonzero tangent vector can be extended to a KVFCL $(X,X')$ with $X\neq 0$.
\end{corollary}
\begin{proof}
We only need to prove the ``only if" part. Suppose $F$ is restrictively CW-homogeneous.
Notice that $\mathcal{K}_{F;2}$ can only cover tangent vectors in a subspace
with positive codimension in each tangent space. For a nonzero tangent vector outside those
subspaces, the existence of the extension follows directly from the restrictive
CW-homogeneity of $F$. Now consider an arbitrary  nonzero $u\in \mathfrak{g}'\subset T_e G$.
One can find a sequence of tangent vectors $ u_n\in T_e G$ with
$ u_n\notin\mathfrak{g}'$, $\forall n$, such that $ u=\mathop{\lim}\limits_{n\rightarrow \infty} u_n$.
Each tangent vector $ u_n$ can be extended to a KVFCL $(X_n,X'_n)\in\mathcal{K}_{F;1}$.
By taking a subsequence, we can have
\begin{equation}
\lim_{n\rightarrow\infty}(X_n,X'_n)=(X,X')\in\mathcal{K}_{F;1}\backslash 0,
\end{equation}
which takes the value $u$ at $e$. By Lemma \ref{lem-4-2}, $X\neq 0$.
For tangent vectors at other points, the argument is similar. This completes the
proof of the corollary.
\end{proof}

The following theorem implies, for each $X\in \mathfrak{g}$, there are not too much
$X'\in\mathfrak{c}(\mathfrak{g}')$ such that $(X,X')\in\mathcal{K}_{F;1}$.
\begin{theorem}\label{thm-4-4}
Keep all notations as above. Suppose both $(X,X')$ and $(X,X'')$ define KVFCLs in
$\mathcal{K}_{F;1}$. Then there exists $c\in\mathbb{R}$ such that $X'-X''=cv$, where
$ v$ is the $\alpha$-dual of $\beta$.
\end{theorem}

\begin{proof}
Lemma \ref{lem-4-2} indicates when $X=0$, we have $X'=X''=0$, so it is obvious. We can
assume $X\neq 0$. The condition that $(X,X')\in\mathcal{K}_{F;1}$ implies
\begin{equation}\label{f-4-10}
\alpha(\mathrm{Ad}(\exp(tY)g)X-X')\phi(\frac{\beta(\mathrm{Ad}(\exp(tY)g)X-X')}{\alpha(\mathrm{Ad}(\exp(tY)g)X-X')})=\mathrm{const}.
\end{equation}
Differentiate (\ref{f-4-10}) with respect to $t$ and take $t=0$, then we get
\begin{equation}\label{f-4-11}
\frac{\phi(s)-s\phi'(s)}{\alpha(\mathrm{Ad}(g)X-X')}\langle[Y,\mathrm{Ad}(g)X],
\mathrm{Ad}(g)X-X'\rangle+\phi'(s)\beta([Y,\mathrm{Ad}(g)X])=0,
\end{equation}
where $s=\beta(\mathrm{Ad}(g)X-X')/\alpha(\mathrm{Ad}(g)X-X')$. A similar
equality holds for $(X,X'')$. Notice $\phi(s)-s\phi'(s)>0$, So for any $g\in G$ and
$Y\in\mathfrak{g}$, if $\beta([Y,\mathrm{Ad}(g)X])=0$, we have
\begin{equation}
\langle [Y,\mathrm{Ad}(g)X],\mathrm{Ad}(g)X-X'\rangle
=\langle[Y,\mathrm{Ad}(g)X],\mathrm{Ad}(g)X-X''\rangle=0,
\end{equation}
and then
\begin{equation}
\langle [Y,\mathrm{Ad}(g)X],X'-X''\rangle=0.
\end{equation}
To finish the proof for the theorem, we only need to prove the set
\begin{equation}
\mathcal{S}=\bigcup_{g\in G}([\mathfrak{g},\mathrm{Ad}(g)X]\cap\mathrm{ker}\beta)
\end{equation}
span the subspace $ V_1=\mathrm{ker}\beta= V'^{\perp_{\mathrm{bi}}}$, in which $ V'$ is
the dual of $\mathrm{ker}\beta$ with respect to the bi-invariant metric. Assume on the
contrary, there is a nonzero $ v''\in V_1$, such that
\begin{equation}
 v''\in([\mathfrak{g},\mathrm{Ad}(g)X]\cap  V_1)^{\perp_{\mathrm{bi}}}
=[\mathfrak{g},\mathrm{Ad}(g)X]^{\perp_{\mathrm{bi}}}+\mathbb{R} v',
\end{equation}
$\forall g\in G$.
So $ v''$ is contained in
\begin{equation}
\bigcap_{g\in G}([\mathfrak{g},\mathrm{Ad}(g)X]^{\perp_\mathrm{bi}}+\mathbb{R} v')
=\bigcap_{g\in G}(\mathrm{Ad}(g)\mathfrak{c}_{\mathfrak{g}}(X)+\mathbb{R} v').
\end{equation}
Notice the nonzero vectors $ v''$ and $ v'$ are linearly independent to each other. Thus
for any $g\in G$, $\mathrm{Ad}(g)\mathfrak{c}_{\mathfrak{g}}(X)$ has a nonzero
intersection with the 2-dimensional real linear space spanned by $ v'$ and $ v''$. The
next lemma states it is impossible.
\end{proof}

\begin{lemma}
Let $G$ be a compact connected simple Lie group with Lie algebra $\mathfrak{g}$, $X
\in\mathfrak{g}$ be a nonzero vector, and $\mathbf{L}\subset\mathfrak{g}$ be a real
subspace with $\dim \mathbf{L}=2$. Then there exists $g\in G$, such that
$\mathbf{L}\cap\mathrm{Ad}(g)\mathfrak{c}_{\mathfrak{g}}(X)=\{0\}$.
\end{lemma}
\begin{proof}
We will prove the lemma by contradiction.  Assume on the contrary
\begin{equation}
\mathbf{L}\cap \mathrm{Ad}(g) \mathfrak{c}_{\mathfrak{g}}(X)\neq \{0\}, \forall g\in G.
\end{equation}
Then the minimum of $\dim\mathbf{L}\cap \mathrm{Ad}(g) \mathfrak{c}_{\mathfrak{g}}(X)$,
$\forall g\in G$, is 1 or 2. If it is 2, then $\mathbf{L}$ is contained in the center of
$\mathfrak{g}$, which is a contradiction. So we can suitably change $X$ by
conjugations, such that $\dim\mathbf{L}\cap\mathfrak{c}_\mathfrak{g}(X)=1$.
By the semi-continuity, for all $g\in G$ sufficiently close to $e$, we also have
$\dim\mathbf{L}\cap \mathrm{Ad}(g)\mathfrak{c}_\mathfrak{g}(X)=1$.

Assume $U\in\mathbf{L}$ generates $\mathbf{L}\cap\mathfrak{c}_\mathfrak{g}(X)$.
Let $\{U,U'\}$ be a basis of $\mathbf{L}$. Then there is a smooth real function
$f(g)$ for $g\in G$ sufficiently close to $e$, such that $f(e)=0$ and
$ \mathbf{L}\cap \mathrm{Ad}(g)\mathfrak{c}_\mathfrak{g}(X)$ is generated by
$U+f(g)U'$.

Take $g=\exp(tY)$, and differentiate
$[U+f(g)U',\mathrm{Ad}(g)X]=0$
with respect to $t$ at $t=0$, we have
\begin{equation}
[U,[Y,X]]+Df(Y)[U',X]=0, \forall Y\in\mathfrak{g},
\end{equation}
in which $Df:\mathfrak{g}\rightarrow\mathbb{R}$ is the differential of $f$ at $e$. Thus $\dim [U,[X,\mathfrak{g}]]\leq 1$.
Because $\dim [U,[X,\mathfrak{g}]]$ is an even number, so it must be 0.
We have $[U,[X,\mathfrak{g}]]=0$
and $[U',X]\neq 0$, then $Df\equiv 0$. If we change $X$ to $\mathrm{Ad}(g)X$, in which
$g\in G$ is sufficiently close to $e$, we can get the same property for the corresponding
differential. This implies $f\equiv 0$ around $e$, i.e. $[U,\mathrm{Ad}(g)X]=0$
for $g$  sufficiently close to $e$. This happens only when $U\in\mathfrak{c}
(\mathfrak{g})$, which is a contradiction.
\end{proof}

\subsection{The properties of $\mathcal{K}_{F;1}$ when $F$ is restrictive CW-homogeneous}

We keep all notations as before, and further assume $F$ is restrictively
CW-homogeneous. Then the set $\mathcal{K}_{F;1}$ of KVFCLs
satisfies the following properties.

\begin{lemma}\label{lem-4-6}
Keep all notations as before and assume $F$ is restrictively CW-homo-geneous, then we have

(1) The function $\phi$ is real analytic on $[-1,1]$.

(2) The subset $\mathcal{K}_{F;1}\backslash \{0\}\subset (\mathfrak{g}\oplus\mathfrak{g}')\backslash \{0\}$ is a closed real analytic
subvariety.

(3) For any $X\in \mathfrak{g}$, there are at most finite different $X'$s,
such that $(X,X')\in\mathcal{K}_{F;1}$.
\end{lemma}

\begin{proof}
(1) By Corollary \ref{cor-4-3}, for any $s_0\in [-1,1]$, we can find a tangent
vector $ u$ with $F( u)=1$, such that $s_0=\beta( u)/\alpha( u)$, and the tangent vector
$ u$ can be extended to a Killing vector field $(X,X')\in\mathcal{K}_{F;1}$ with $X\neq 0$, i.e.
\begin{equation}\label{f-4-19}
\alpha(\mathrm{Ad}(g)X-X')\phi(\frac{\beta(\mathrm{Ad}(g)X-X')}
{\alpha(\mathrm{Ad}(g)X-X')})=1,\forall g\in G.
\end{equation}
The function $s(g)=\beta(\mathrm{Ad}(g)X-X')/\alpha(\mathrm{Ad}(g)X-X')$ can not
be a constant function for $g\in G$. Otherwise by (\ref{f-4-19}),
$\beta(\mathrm{Ad}(g)X-X')$ is a constant function for $g\in G$. Then the
$\mathrm{Ad}(G)$-orbit $\mathcal{O}_X$ is contained in a flat hyperplane,
which is a contradiction with the assumption that $X\neq 0$ and  $G$ is simple. We will
denote the range of $s(g)$ as $\mathcal{I}_{(X,X')}$, which is a closed interval.

For any side of $s_0$ which is contained in $\mathcal{I}_{(X,X')}$, the
positive side for example, we can choose $X$ within the orbit $\mathcal{O}_X$
and find a vector $Y\in\mathfrak{g}$ such that 
the real analytic function $f(t)=s(\exp(tY))$ satisfies for some $k\in\mathbb{N}$,
\begin{equation}\label{f-4-20}
f(0)=s_0,f'(0)=f''(0)=\cdots=f^{(k-1)}(0)=0,f^{(k)}(0)>0.
\end{equation}
We can find a suitable real analytic change of variable $\tilde{t}=\tilde{t}(t)$,
$\tilde{t}(0)=0$, such that $f(t)=f(0)+\tilde{t}^k$ around $\tilde{t}=0$.
The equality (\ref{f-4-19}) can also be written as
\begin{equation}\label{f-4-21}
\phi(s_0+\tilde{t}^k)=\frac{1}{\alpha(\mathrm{Ad}(\exp(tY))X-X')}.
\end{equation}
Around $\tilde{t}=0$, the left side of (\ref{f-4-21}) is a smooth function of
$\tilde{t}$, which derivatives with respect to $\tilde{t}$ at $\tilde{t}=0$
vanishes except those with $k$-multiple degrees, and then the right side is
a real analytic function of $\tilde{t}$ with the same properties for its
derivatives at $\tilde{t}=0$. Thus the right side is a real analytic function
of $\bar{t}=f(t)$ at the positive side of $f(0)=s_0$, and so does $\phi(s)$
at the positive side of $s_0$. The proof for the negative side of $s_0$ is
similar, we just need to require $f^{(k)}(0)<0$ in (\ref{f-4-20}) and take
$\bar{t}=f(t)=f(0)-\tilde{t}^k$.

If $s_0$ is an endpoint of $\mathcal{I}_{(X,X')}$, the argument above guarantees
$\phi(s)$ is real analytic at one side of $s_0$. We will see how to use
Lemma \ref{lem-4-2} to prove the real analytic property of $\phi(s)$ for the other
side, the negative side for example. If there is another Killing vector field
$(X_0,X'_0)$ in $\mathcal{K}_{F;1}$ such that an open neighborhood of $s_0$ is
contained in $\mathcal{I}_{(X_0,X'_0)}$, then it is done. Otherwise we
can find a sequence $s_n$ approaching $s_0$ from below. For each $s_n$, we
can find a KVFCL $(X_n,X'_n)\in\mathcal{K}_{F;1}$ with length 1, such that
$s_n$ is contained in $\mathcal{I}_{(X_n,X'_n)}$ which lies below $s_0$. By taking
a subsequence, this sequence of KVFCLs converges to a KVFCL $(X_0,X'_0)\in
\mathcal{K}_{F;1}\backslash \{0\}$, such that $\mathcal{I}_{(X_0,X'_0)}$ contains the negative
side of the endpoint $s_0$.

To summarize, the smooth function $\phi(s)$ is real analytic for both sides of
each point in $[-1,1]$, so it is a real analytic function on $[-1,1]$.

(2) Around any $(X,X')\in\mathcal{K}_{F;1}\backslash \{0\}$, the equations defining
$\mathcal{K}_{F;1}$ can be presented as
\begin{equation}
\alpha(\mathrm{Ad(g)}X-X')\phi(\frac{\beta(\mathrm{Ad}(g)X-X')}
{\alpha(\mathrm{Ad}(g)X-X')})=\alpha(X-X')\phi(\frac{\beta(X-X')}{\alpha(X-X')}),
\forall g\in G,
\end{equation}
which are real analytic equations for $X$ and $X'$, because $\phi$ is real
analytic on $[-1,1]$. So $\mathcal{K}_{F;1}\backslash \{0\}$ is a
closed real analytic subvariety of $(\mathfrak{g}\oplus\mathfrak{g}')\backslash \{0\}$.

(3) When $X=0$, the assertion follows Lemma \ref{lem-4-2} directly.
Now assume $X\neq 0$. If on the contrary there are a sequence of different
$X'_n$s such that $(X,X'_n)\in\mathcal{K}_{F;1}$. By Lemma \ref{lem-4-2}, taking
a subsequence if necessary, we can assume $\lim_{n\rightarrow\infty}X'_n=X'$.
By  Theorem \ref{thm-4-4}, there is a sequence $\{t_n\}\subset\mathbb{R}\backslash \{0\}$, such that $\lim_{n\rightarrow\infty}t_n=0$, and $X'_n=X'-t_n v$. So we have
\begin{equation}
F(\mathrm{Ad}(g)X-X'+t_n  v)\equiv C_n,\forall g\in G.
\end{equation}
Since $\phi$ is real
analytic, the continuous function $F(\mathrm{Ad}(g)X-X'+tv)$ of $g$ and $t$
is real analytic whenever $\mathrm{Ad}(g)X-X'+tv\neq 0$.
Because $(X,X')\in\mathcal{K}_{F;1}\backslash\{0\}$, for $t$ sufficiently close to
0, $F(\mathrm{Ad}(g)X-X'+tv)\neq 0$ for all $g\in G$, and then $F(\mathrm{Ad}(g)X-X'+tv)$ is a constant function of $g$. If there is a number
 $t_0=\mathrm{inf}\{t>0|F(\mathrm{Ad}(g)X-X'+tv)=0,\mbox{ for some }g\in G\}$, then $t_0>0$ and
there is $g_0\in G$ such that $\mathrm{Ad}(g_0)X-X'+t_0 v=0$.
For any $t\in [0,t_0)$,
$F(\mathrm{Ad}(g)X-X'+tv)$ is a constant function of $g$. By the continuity,
\begin{equation}
\mathrm{Ad}(g)X-X'+t_0 v=0, \forall g\in G,
\end{equation}
i.e. the $\mathrm{Ad}(G)$-orbit $\mathcal{O}_X$ is contained in a line, which
is impossible when $X\neq 0$ and $G$ is compact simple. So $\mathrm{Ad}(g)X-X'+tv\neq 0$ and $(X,X'-tv)\in\mathcal{K}_{F;1}$ for all $t\geq 0$. This is a contradiction with Theorem \ref{thm-4-4}.
\end{proof}

There are two natural projections from $\mathcal{K}_{F;1}\backslash \{0\}$
to $\mathfrak{g}\backslash \{0\}$, namely,
\begin{equation}
\pi_1(X,X')=X-X',\mbox{ and }\pi_2(X,X')=X.
\end{equation}
The first projection maps each Killing vector field to its value at $e$. When
$(G,F)$ is restrictive CW-homogeneous, by Corollary \ref{cor-4-3}, the map
$\pi_1$ is surjective. Thus the dimension of the real analytic variety
$\mathcal{K}_{F;1}\backslash \{0\}$ is no less than $\dim\mathfrak{g}$.
Lemma \ref{lem-4-6} indicates the map $\pi_2$ has a finite pre-image for each $X$,
which implies the dimension of $\mathcal{K}_{F;1}\backslash \{0\}$ must be exactly $\dim\mathfrak{g}$. Whitney's theorem on the local stratifiacation of analytic
varieties \cite{WH65} indicates locally $\mathcal{K}_{F;1}\backslash \{0\}$
can be decomposed as the disjoint union of finite smooth manifolds, among which
there is one with the same dimension as $\mathfrak{g}$. Restricted to this
subset, the finite map $\pi_2$ must be regular somewhere. So $\pi_2(\mathcal{K}_{F;1}\backslash \{0\})$ contains a nonempty open subset
$\mathcal{U}\subset
\mathfrak{g}\backslash \{0\}$. The $\mathrm{Ad}(G)$-actions on the first factor
preserve $\mathcal{K}_{F;1}\backslash\{0\}$, so we can assume $\mathcal{U}$ is
an $\mathrm{Ad}(G)$-invariant nonempty open subset of $\mathfrak{g}\backslash \{0\}$.

Let $\mathfrak{t}$ be any Cartan subalgebra of $\mathfrak{g}$. Then $\mathcal{U}'=\mathcal{U}
\cap\mathfrak{t}$ is a nonempty open subset of $\mathfrak{t}$. For any nonzero
$X$ in $\mathcal{U}'$, there is a $X'\in\mathfrak{c}(\mathfrak{g}')$, such that
$(X,X')\in\mathcal{K}_{F;1}$. Let $ V_1=\mathrm{ker}\beta$, and
$\mathrm{pr}_1$ be the $\alpha$-orthogonal projection from $\mathfrak{g}$ to $ V_1$.
Though there maybe many choices for $X'$, they have the same $\mathrm{pr}_1 X'$
by Theorem \ref{thm-4-4}.

From (\ref{f-4-11}), we have seen $l(X)=\mathrm{pr}_1 X'$ for $X\in\mathcal{U}'$
satisfies the following condition
\begin{equation}\label{f-4-26}
\langle [\mathfrak{g},\mathrm{Ad}(g)X]\cap V_1,
\mathrm{Ad}(g)X-l(X)\rangle=0.
\end{equation}

Now we will see $l(X)$ can be extended to a linear map on $\mathfrak{t}$ with
(\ref{f-4-26}) satisfied.

Choose a basis $\{X_1,\ldots,X_m\}$ of $\mathfrak{t}$ from the regular vectors
in $\mathcal{U}'$. For each $X_i$, there is $X''_i=l(X_i)$ such that
$\langle [\mathfrak{g},\mathrm{Ad}(g)X_i]\cap V_1,\mathrm{Ad}(g)X_i-X''_i\rangle=0$. For $X=\sum_{i=1}^m c_i X_i$, let
$X''=\sum_{i=1}^m c_iX''_i$. Because $[\mathfrak{g},\mathrm{Ad}(g)X]\subset
[\mathfrak{g},\mathrm{Ad}(g)X_i]$, $\forall i$, we have
\begin{equation}
\langle [\mathfrak{g},\mathrm{Ad}(g)X]\cap V_1,\mathrm{Ad}(g)X_i-X''_i\rangle=0.
\end{equation}
Take the linear combination of the above
equalities for each $i$, we get
\begin{equation}
\langle [\mathfrak{g},\mathrm{Ad}(g)X]\cap V_1,\mathrm{Ad}(g)X-X''\rangle=0.
\end{equation}
This defines a linear map from $X$
to $X''$, satisfying (\ref{f-4-26}). From the proof of Theorem \ref{thm-4-4}, this
linear map coincides with $\mathrm{pr}_1$ when $(X,X')\in\mathcal{K}_{F;1}$.

For any $X_1,X_2\in\pi_2(\mathcal{K}_{F;1})\cap(\mathfrak{t}\backslash \{0\})$ in the same orbit of Weyl group actions, they share the same $X'$ such that
$(X_1,X'),(X_2,X')\in\mathcal{K}_{F;1}$. So $l(X_1)=l(X_2)$, i.e. the linear map
$l$ on $\mathfrak{t}$ is invariant for the Weyl group actions, which must be the
0 map.

Change $t$ arbitrarily, then we have
\begin{lemma}\label{lem-4-7}
Keep all notations of this subsection. For any $(X,X')\in\mathcal{K}_{F;1}$, we
have $X'$ is a scalar multiple of $ v$, the $\alpha$-dual of $\beta$.
\end{lemma}

\subsection{Proof of Theorem \ref{main-theorem} for compact connected simple $G$}
To prove the main theorem for compact connected simple $G$, we only need to
consider non-Riemannian metrics. So we will assume $F$ is a restrictive
CW-homogeneous
left invariant non-Riemannian $(\alpha,\beta)$-metric on a compact connected
simple Lie group $G$, and keep all notations as before. We will see how the
properties of $\mathcal{K}_{F;1}$ can determine the metric $\alpha$ and help
us prove the theorem.

For any nonzero $X\in\mathcal{U}$, we can find a pair $(X,X')\in\mathcal{K}_{F;1}$.
We have just proven $X'$ is a scalar multiple of $ V$. The equality (\ref{f-4-11}),
with $g=e$, indicates the following condition is satisfied,
\begin{equation}\label{f-4-29}
\exists c\in \mathbb{R}, \mbox{ such that }\langle[Y,X],X-cv\rangle=0,
\forall Y\in\mathfrak{g}.
\end{equation}
In fact $c$ can be determined by
\begin{equation}
cv=\frac{\phi(s)-s\phi'(s)}{\alpha(X-X')}X'-\phi'(s) v,
\end{equation}
in which $s=\beta(X-X')/\alpha(X-X')$.

The next lemma indicates (\ref{f-4-29}) can be satisfied for all $X\in \mathfrak{g}$, and it can define a linear function.

\begin{lemma}\label{lem-4-8}
Keep all notations as before. Then there is a linear function $c(\cdot):\mathfrak{g}\rightarrow\mathbb{R}$, such that
\begin{equation}\label{f-4-31}
\langle [Y,X],X-c(X) v\rangle=0, \forall Y\in\mathfrak{g}
\end{equation}
\end{lemma}

\begin{proof}
We will first construct a function $c(X)$ satisfying (\ref{f-4-31}). Then we
will further refine it to be linear.

For any $X\in\mathfrak{g}$, let $\mathfrak{t}$ be a Cartan subalgebra containing $X$,
and $\{X_,\ldots,X_m\}$ a basis of $\mathfrak{t}$, in which each $X_i$ is a regular
vector in $\mathcal{U}\cap\mathfrak{t}$. For each $X_i$, the corresponding
$c_i=c(X_i)$ indicated by the lemma can be found. Assume $X=\sum_{i=1}^m a_i X_i$,
then take $c=\sum_{i=1}^m a_i c_i$. Because $[\mathfrak{g},X]\subset
[\mathfrak{g},X_i]$ for each $i$, for any $Y\in\mathfrak{g}$, we can
find $Y_i\in\mathfrak{g}$ such that $[Y,X]=[Y_i,X_i]$, so we have
\begin{equation}\label{f-4-40}
\langle [Y,X],X_i-c_i  v\rangle=\langle [Y_i,X_i],X_i-c_i  v\rangle=0.
\end{equation}
Take the linear combination of
 (\ref{f-4-40}) for each $i$, we see the constant $c$ given above
satisfies (\ref{f-4-31}) for $X$, which can define the function $c(X)$ on
$\mathfrak{g}$. If $X$ is contained by more than one Cartan subalgebra,
and there are different $c_1$ and $c_2$ such that
\begin{equation}
\langle [Y,X],X-c_1 v\rangle=\langle [Y,X],X-c_2 v\rangle=0,\forall Y\in\mathfrak{g},
\end{equation}
then it is easy to see $c(X)=0$ satisfies (\ref{f-4-31}).

Denote the dual of $\beta$ with respect to the bi-invariant metric as $ v'$.
Let $Y_0\in\mathfrak{g}$ be any vector with $[Y_0, v']\neq 0$, or equivalently
$\langle[Y_0,X], V\rangle\neq 0$ for some $X$.
From (\ref{f-4-31}),
the function $f_0(X)=\langle [Y_0,X],X\rangle$ vanishes on the
codimension 1 linear subspace
\begin{equation}
\{X|\langle [Y_0,X], V\rangle=0\}\subset\mathfrak{g}.
\end{equation}
This can only happen when $f_0(X)$ splits as the product of two linear factors.
Up to scalar multiplications, one
is $\langle [Y_0,X], V\rangle$, and the other is $\tilde{c}(X)$ which coincides with $c(X)$ on the nonempty
open subset
\begin{equation}
\{X|\langle [Y_0,X], V\rangle\neq 0\}\subset\mathfrak{g}.
\end{equation}
For $X$ in this open set, we have
\begin{equation}
\langle [Y,X],X\rangle=\langle [Y,X],\tilde{c}(X) V\rangle, \forall Y\in\mathfrak{g},
\end{equation}
so it is still valid for all $X,Y\in\mathfrak{g}$.
With $c(X)$ changed to $\tilde{c}(X)$, we have finished the proof for the
lemma.
\end{proof}

Let $l_0:\mathfrak{g}\rightarrow\mathfrak{g}$ be the linear isomorphism defined
by $\langle X,Y\rangle=\langle X,l_0(Y)\rangle_{\mathrm{bi}}$, and
$f:\mathfrak{g}\times\mathfrak{g}\rightarrow\mathbb{R}$ the bi-linear function defined by
\begin{equation}
f(X,Y)=\langle X-c(X) V,Y\rangle=\langle l_0(X-c(X) V),Y\rangle_{\mathrm{bi}},
\end{equation}
in which $c(\cdot)$ is the linear function indicated by Lemma \ref{lem-4-8}.
Let $l_1(X)=l_0(X-c(X) V)$, then (\ref{f-4-31}) indicates $l_1$ maps the
regular vectors in any Cartan subalgebra $\mathfrak{t}$ back to $\mathfrak{t}$
itself. So it preserves each Cartan subalgebra.
There
is a nonzero vector $X\in\mathfrak{g}$, such that $\mathbb{R}X$ is the intersection
of some Cartan subalgebras of $\mathfrak{g}$. Then any vector on the
$\mathrm{Ad}(G)$-orbit $\mathcal{O}_X$ is an eigenvector of $l_1$. Because
$X\neq 0$ and $\mathfrak{g}$ is compact simple. This can only happen when $l_1$
is a scalar multiple of the identity map. So $f(X,Y)$ is a bi-invariant inner
product on $\mathfrak{g}$.

Choose $(X,Y)\in V_1\times V_1$, or $(X,Y)\in V_2\times
 V_1$, in which $ V_1=\mathrm{ker}\beta$ and
$ V_2=\mathbb{R} V$ for $f(X,Y)$, we see immediately
$ V_1$ and $ V_2$ are orthogonal with respect to both inner
products from $\alpha$ and the bi-invariant metric, and restricted to
$ V_1$, $\alpha$ only differs from the bi-invariant metric by a scalar
multiplication. To summarize we have
\begin{lemma}\label{lem-4-9}
Keep all notations as above, then there are constants $a$ and $b$, such that
$\alpha^2(X)=a||X||^2_{\mathrm{bi}}+b\beta^2(X)$.
\end{lemma}

By Lemma \ref{lem-4-9}, the $(\alpha,\beta)$-norm $F$ on $\mathfrak{g}$
can also be presented
as $F=\tilde{\alpha}\tilde{\phi}(\tilde{\beta}/\tilde{\alpha})$, in which
$\tilde{\alpha}$ is the bi-invariant metric  with $\alpha^2(X)=a
\tilde{\alpha}^2(X)+b\beta^2(X)$, $\tilde{\beta}=\beta$, and
$\tilde{\phi}(s)=\sqrt{a+bs^2}\phi(s/\sqrt{a+bs^2})$.

If there is a KVFCL of the form $(X,0)$ with $X\neq 0$, then
we have
\begin{equation}
\tilde{\alpha}(\mathrm{Ad}(g)X)\tilde{\phi}(\frac{{\beta}(\mathrm{Ad}(g))}
{\tilde{\alpha}(\mathrm{Ad}(g)X)})=\mathrm{const}, \forall g\in G.
\end{equation}
Since $\tilde{\alpha}(\mathrm{Ad}(g)X)=\tilde{\alpha}(X)$ is a nonzero constant
function of $g$, and ${\beta}(\mathrm{Ad}(g)X)$ is not a constant function, the
real analytic function $\tilde{\phi}$ must be a constant function, i.e. the left invariant metric $F$ must
be a bi-invariant Riemannian metric.
Though it is a contradiction with the assumption that $F$ is non-Riemannian, it
helps us with the discussion in the next case.

If there is a KVFCL of the form $(X,\lambda  V)$ with $X\neq 0$ and $\lambda\neq 0$,
and assume its $F$-length function is constantly 1, i.e.
\begin{equation}\label{f-4-37}
F(\mathrm{Ad}(g)X-\lambda  V)=1, \forall g\in G.
\end{equation}
The strong convexity of $F$ implies $F(-\lambda  V)<1$. Applying
a navigation transformation to $F$ which set the origin at $-\lambda  V$, we
get a new left invariant $(\alpha,\beta)$-metric $F'$. The $(\alpha,\beta)$-norm in
$\mathfrak{g}$ defined by $F'$ is also denoted as $F'$. In $T_e G=\mathfrak{g}$, the indicatrix of $F'$ is a parallel shift of that of $F$,
with $-\lambda  V$ shifted to 0. While presenting $F'$, we can keep $\alpha$ and $\beta$ in
the good normalized datum for $F$ and just change the function $\phi$. So any isometry of $(G,F)$, which preserves $\alpha$ and $\beta$, is also an isometry of $(G,F')$, and any Killing vector field of $(G,F)$ is still a Killing vector
field of $(G,F')$. Because of (\ref{f-4-37}),
\begin{equation}
F'(\mathrm{Ad}(g)X)=1, \forall g\in G,
\end{equation}
so $(X,0)$ defines Killing vector field of constant length 1 for $F'$.

There is
another presentation of $F'$ using the bi-invariant $\tilde{\alpha}$ and $\tilde{\beta}=\beta$.
By the discussion in the last case, $F'$ is a bi-invariant Riemannian metric and
then $F$ is a Randers metric. This finishes the proof of Theorem \ref{main-theorem}
in the case that $G$ is a compact connected simple Lie group.

\section{Proof of Theorem \ref{main-theorem} for a compact connected
semi-simple $G$}
\subsection{Notations and assumptions}
Let $G$ be a compact connected semi-simple Lie group and $F$ be a left invariant
restrictively CW-homogeneous $(\alpha,\beta)$-metric on $G$. There is no harm that
we assume $F$ is non-Riemannian, otherwise the main theorem needs no proof.

When $G$ is not simply-connected, there is a connected simply-connected
$\tilde{G}$ covering $G$. Let $\tilde{F}$ be the induced metric on $\tilde{G}$, then
$\tilde{F}$ is also a non-Riemannian left invariant $(\alpha,\beta)$-metric.
Any KVFCL for $(G,F)$ induces a KVFCL for $(\tilde{G},\tilde{F})$ which exhausts all
tangent vectors of $\tilde{G}$ as well as $G$. So by Proposition \ref{pro-2-5}
$\tilde{F}$ is also restrictively CW-homogeneous. We only need to prove $\tilde{F}$
is Randers, then so does $F$. So we will further assume $G$ is simply-connected.

We wish $I_0(G,F)$ be contained by $L(G)R(G)$, then we can have an explicit
description of Killing vector fields, study the set of all KVFCLs, and then the
restrictive CW-homogeneity. Though it may not be correct when we consider semi-simple $G$ rather than the simple ones, we
can change $F$ to $F'$ by a diffeomorphism, such that
\begin{equation}
L(G)\subset I_0(G,F')\subset L(G)R(G)
\end{equation}
is satisfied.

\begin{lemma}
Let $F$ be a left invariant Finsler metric on the compact connected simply-connected group $G$, then there is a diffeomorphism $f$ on $G$, such that $F'=f^* F$ satisfies
$L(G)\subset I_0(G,F')\subset L(G)R(G)$.
\end{lemma}

\begin{proof}
The proof is very similar to the one in the Riemannian case. Let $\mathfrak{k}$
be the Lie algebra of $I_0(G,F)$, then we have a linear space decomposition
$\mathfrak{k}=\mathfrak{g}+\mathfrak{h}$, in which $\mathfrak{g}$ is in fact the
Lie algebra of $L(G)$. Ozeki's theorem \cite{OZ77} states that we can find an ideal $\mathfrak{g}'$
of $\mathfrak{k}$ which is isometric to $\mathfrak{g}$ and $\mathfrak{g}'\cap\mathfrak{h}=0$. Let $G'$ be the
subgroup of $I_0(G,F)$ corresponding to $\mathfrak{g}'$. It acts transitively on $G$.
The map $\tilde{f}:G'\rightarrow G$ by $\tilde{f}(g')=g'(e)$, $\forall g'\in G'$ is a covering
map. By the simply-connectedness of $G$, it is a diffeomorphism. The metric $\tilde{f}^*
F$ is left invariant on $G'$. There is a natural group isomorphism
$I_0(G,F)\cong I_0(G',\tilde{f}^* F)$ which relates any $\rho\in I_(G,F)$ with $\tilde{f}^{-1}\rho
\tilde{f}$, under which $G'\subset I_0(G,F)$ is identified with the group of all left
translations on $G'$ in $I_0(G,f^* F)$. So $L(G')$ is normal in $I_0(G',\tilde{f}^* F)$, and then
$I_0(G',\tilde{f}^* F)\in L(G')R(G')$. Let $f$ be the diffeomorphism on $G$ defined by the composition of $\tilde{f}$ with any isomorphism from $G$ to $G'$. Then $L(G)\subset L(G,f^* F)\subset L(G)R(G)$.
\end{proof}

Obviously $F'=f^* F$ is a non-Riemannian metric, an $(\alpha,\beta)$-metric or a Randers metric if and only if
$F$ is respectively. Any KVFCL $X$ of $(G,F)$ one-to-one corresponds to the KVFCL $f_*^{-1} X$ of $(G,f^* F)$, so $F$
is restrictively CW-homogeneous if and only if $F$ is. So we only need to prove the main theorem
with the condition $I_0(G,F)\subset L(G)R(G)$.

To summarize, we only need to prove Theorem
\ref{main-theorem} with the following assumptions:
$G$ is not simple,
$G$ is simply-connected,
$F$ is non-Riemannian, and
$I_0(G,F)\subset L(G)R(G)$.

The space of Killing vector fields for $(G,F)$ can then be presented explicitly.
Let $G'$ be the closed connected subgroup of $G$ such that $R(G')$ is the maximal
connected subgroup of isometric right translations, and
$\mathfrak{g}'=\mathrm{Lie}(G')$, then the Lie algebra of $I_0(G,F)$, i.e.
the space of Killing vector fields for $(G,F)$ is the direct sum
$\mathfrak{g}\oplus\mathfrak{g}'$,
in which $\mathfrak{g}$ corresponds to left translations and $\mathfrak{g}'$ corresponds to isometric right translations. Denote $\beta(u)=\langle u,v\rangle=\langle u,v'\rangle_{\mathrm{bi}}$, $\forall u\in\mathfrak{g}$, then $\mathfrak{g}'$ is a subalgebra of
$\mathfrak{c}_{\mathfrak{g}}(v)$ and $\mathfrak{c}_{\mathfrak{g}}(v')$.

Evaluation of the $F$-length function of a Killing vector field
$(X,X')\in\mathfrak{g}\oplus\mathfrak{g}'$ at the point $g''=gg'^{-1}\in G$, in which $g\in G$ and $g'\in G'$,
is $F(\mathrm{Ad}(g) X-\mathrm{Ad}(g')X')$. Because $F$ is $\mathrm{Ad}(G')$-invariant in $\mathfrak{g}=TG_e$, there
is no contradiction when we use different $g$ and $g'$.

\subsection{Finishing the Proof of Theorem \ref{main-theorem}}
We keep all notations and assumptions as in the last subsection.

Let 
$\mathfrak{g}=\mathfrak{g}_1\oplus\mathfrak{g}_2$ be any direct sum of nontrivial ideals
 and correspondingly $G=G_1\times G_2$ the product of closed subgroups.
On $G_1$ (or $G_2$), the metric $F$
induces a left invariant $(\alpha,\beta)$-metric $F|_{G_1}$. When $F$ is restrictively
CW-homogeneous, then its restriction on $G_1$ is also restrictively CW-homogeneous. To see this, choose any nonzero tangent
vector $X''_1\in \mathfrak{g}_1={T_e G_1}$, we can extended it to a KVFCL of $(G,F)$, defined by
\begin{equation}
X=(X_1,X_2,X'_1,X'_2)\in\mathfrak{g}\oplus\mathfrak{g}'\subset\mathfrak{g}_1\oplus
\mathfrak{g}_2\oplus\mathfrak{g}_1\oplus\mathfrak{g}_2,
\end{equation}
in which $X_1$ and $X_2$ are for left translations, and $X'_1$ and $X'_2$ are for right translations, on $G_1$ and $G_2$ respectively, $X_1-X'_1=X''_1$, and $X_2=X'_2$. It is of constant length implies for any $(g'_1,g'_2)\in G'$, and any $(g_1,g_2)\in G$ with $g_2=g'_2$, we have
\begin{equation}
F((\mathrm{Ad}({g_1})X_1-\mathrm{Ad}({g'_1})X'_1,\mathrm{Ad}({g_2})X_2-\mathrm{Ad}({g'_2})X'_2))=F((\mathrm{Ad}({g_1})X_1-\mathrm{Ad}({g'_1})X'_1,0))\nonumber\\
\end{equation}
is a constant function of $g_1$ and $g'_1$. So $(X_1,X'_1)$ defines a KVFCL
of $(G_1, F|_{G_1})$. It can exhaust all tangent vectors $X''_1=X_1-X'_1$, so
the restriction of $F$ to $G_1$ is restrictively CW-homogeneous. We have proven in
the last section that $F|_{G_1}$ is Randers. Let $(\phi,\alpha,\beta)$
be a good normalized datum for $F$, then only the values of $\phi$ for $s\in [-1,1]$ are
used to define $F$. 

Let $v=v_1+v_2$ be the decomposition of the $\alpha$-dual of $\beta$ with respect to
the decomposition of $\mathfrak{g}$. If $v_1\neq 0$ and $v_2=0$, then $(\phi,\alpha|_{G_1},\beta|_{G_1})$
is a good normalized datum of $F|_{G_1}$, i.e. the pointwise norms 
$||\beta|_{G_1}||_{\alpha|_{G_1}}(\cdot)$ are constantly 1, and all values of $\phi$
for $s\in [-1,1]$ are used to define $F|_{G_1}$. When $F|_{G_1}$ is Randers, 
$\phi(s)=\sqrt{k_1+k_2 s^2}+k_3$ for some contants $k_1$, $k_2$ and $k_3$, $\forall s\in [-1,1]$, i.e. the same
function $\phi$ defines a Randers metric $F$ on $G$.

By the observations above, we can prove Theorem $\ref{main-theorem}$ for semi-simples $G$ by
mathematical induction.
As we have mentioned above, we can assume $G$ is a compact connected simply connected Lie group, and the left invariant restrictively CW-homogeneous $F$ is non-Riemannian.

Let $G=G_1\times G_2\times\cdots\times G_n$, in which all $G_i$s are nontrivial simple Lie groups. Correspondingly we have the direct sum decomposition $\mathfrak{g}=\mathfrak{g}_1\oplus
\cdots\oplus\mathfrak{g}_n$ for the Lie algebra.

When $n=1$, i.e. $G$ is simple, we have proven $F$ is a Randers metric in the last section.
Assume we can prove $F$ is a Randers metric when $n= k$,
then we need to prove $F$ is a Randers metric for $n=k+1>1$.

As we have argued, we only need to prove the statement with the assumption $F$ is non-Riemannian and
$L(G)\subset I_0(G,F)\subset L(G)R(G)$. The space of Killing vector fields
of $(G,F)$ can be identified as a direct sum of Lie algebras $\mathfrak{g}\oplus\mathfrak{g}'$. Let $v=v_1+\cdots+v_n$
with respect to the decomposition of $\mathfrak{g}$, then 
\begin{equation}\label{f-5-43}
\mathfrak{g}'\subset \mathfrak{c}_{\mathfrak{g}}(v)=\oplus_{i=1}^n 
\mathfrak{c}_{\mathfrak{g}_i}(v_i).
\end{equation}

 By the inductive assumption,  $F|_{G_1\times\cdots\times G_{n-1}}$ is Randers. If any $v_i=0$, for example $v_n=0$, by the above argument,
 we have seen the metric $F$ on $G$ is also Randers which
finished mathematical induction. Now we will assume $v_i\neq 0$, $\forall i=1,2,\ldots,n$.
Then we have

\begin{lemma}
Keep all notations and assumptions for $G$ and $F$ as above, let $(\phi,\alpha,\beta)$
be a good normalized datum of $F$, and assume $v=v_1+\cdots+v_n$, $v_i\neq 0$, $\forall i=1,2,\ldots,n$
for the $\alpha$-dual $v$ of $\beta$ then the function $\phi$ is real analytic
in $(-1,1)$.
\end{lemma}

\begin{proof}
Let $\mathfrak{g}_0\neq \mathfrak{g}$ be the largest ideal of $\mathfrak{g}$ contained by $\mathrm{ker}\beta$.
For any KVFCL given by $(X,X')\in\mathfrak{g}\oplus\mathfrak{g}'$, with $X\notin \mathfrak{g}_0$,
the range $\mathcal{I}_{(X,X')}$ of
\begin{equation}
s(g)=\beta(\mathrm{Ad}(g)X-X')/\alpha(\mathrm{Ad}(g)X-X'),\forall g\in G,
\end{equation}
is a closed interval with positive length. Otherwise, $s(g)$ is a constant function
of $g$, and then so is $\beta(\mathrm{Ad}(g)X-X')$. It implies the ideal generated by $[X,\mathfrak{g}]$ is contained in $\mathrm{ker}\beta$, which is a contradiction with that
$X\notin\mathfrak{g}_0$.

Consider the open subset
\begin{equation}
\mathcal{U}=\mathfrak{g}\backslash
(\mathfrak{g}_0+\mathfrak{g}')
\end{equation}
in $\mathfrak{g}$. By (\ref{f-5-43}), 
$\mathfrak{g}_0+\mathfrak{g}'$ has a codimensions bigger than 1 in $\mathfrak{g}$, so
$\mathcal{U}$ is a connected dense open subset of $\mathfrak{g}$.
For any $s_0\in (-1,1)$, we can find a tangent
vector $ u\in \mathcal{U}\subset T_e G$ such that $\beta( u)/\alpha( u)=s_0$.
Let $(X,X')$ be a KVFCL which value at $e$ is $ u$, then $X\neq 0$ and $X'\in \mathfrak{g}'$.
Because $X-X'= u\notin\mathfrak{g}_0+\mathfrak{g}'$, we have $X\notin\mathfrak{g}_0$, i.e. $\mathcal{I}_{(X,X')}$ is a
closed interval with a positive length.
Using only these KVFCLs, the proof can be carried out exactly as the one for (1) of Lemma \ref{lem-4-6}.
\end{proof}

The next lemma indicates the real analytic property of $\phi$ guarantees $\phi$ defines
a Randers norm in $\mathfrak{g}=T_e G$, and by the homogeneity of $(G,F)$, finishes
the mathematical induction. 

\begin{lemma}
Let $F$ be an $(\alpha,\beta)$-norm on $ V$, $\dim V=n>2$, which
is non-Riemannian. Assume $(\phi,\alpha,\beta)$
is a normalized datum defining $F$, in which $\phi\in C^\infty[-1,1]$ is real analytic
on $(-1,1)$, and there is a linear subspace $ V'\in V$,
$\dim  V'=m>1$, such that $ V'$ is not contained by $\mbox{ker}\beta$
and the restriction of $F$ in $ V'$ is Randers, then the norm $F$
on $ V$ is Randers.
\end{lemma}

\begin{proof}
Let the restrictions of $F$, $\alpha$ and $\beta$ in $ V'$ be $\tilde{F}$,
$\tilde{\alpha}$ and $\tilde{\beta}$ respectively.
Because $\mbox{ker}\beta$ does not contain $ V'$, i.e. $\tilde{\beta}$
is not constantly 0 on $ V'$. Direct calculation indicates if
$\tilde{F}=\tilde{\alpha}\phi(\tilde{\beta}/\tilde{\alpha})$
is Randers, then $\phi(s)=\sqrt{k_1+k_2 s^2}+k_3 s$ around $s=0$, for some constants $k_1$, $k_2$ and $k_3$. Because
$\phi$ is real analytic on $(-1,1)$ and smooth on $[-1,1]$,
it must satisfy the same formula for all $s\in [-1,1]$, which implies $F$ is Randers.
\end{proof}


\begin{thebibliography}{99}

\bibitem[AW76]{AW76} R. Azencott, E. Wilson, Homogeneous manifolds with negative
curvature I, Trans. Amer. Math. Soc., 215 (1976), 323-362.

\bibitem[BCS00]{BCS00} D. Bao, S. S. Chern and Z. Shen, An introduction to
Riemann-Finsler Geometry, Springer-Verlag, New York, 2000.

\bibitem[BE88]{BE88} V. N. Berestovskii, Homogeneous manifolds with intrinsic
metric I, Siber. Math. Jour., 29 (1988), 887-897.

\bibitem[BN08-1]{BN08-1} V. N. Berestovskii and Yu. G. Nikonorov, Killing vector fields of
constant length on locally symmetric Riemannian manifolds, Transformation Groups,
13 (2008), 25-45.

\bibitem[BN08-2]{BN08-2} V. N. Berestovskii and Yu. G. Nikonorov, On $\delta$-homogeneous
Riemannian manifolds, Diff. Geom. Appl., 26 (2008), 514-535.

\bibitem[BN09]{BN09} V. N. Berestovskii and Yu. G. Nikonorov, Clifford-Wolf homogeneous
Riemannian manifolds, Jour. Differ. Geom., 82 (2009), 467-500.

\bibitem[BP99]{BP99} V. N. Berestovskii and C. Plaut, Homogeneous spaces of
curvature bounded below, Jour. Geom. Anal., 9 (1999), 203-219.

\bibitem[CS05]{CS05} S. S. Chern and Z. Shen, Riemann-Finsler Geometry, World
Scientific Publishers, 2004.

\bibitem[DH02]{DH02} S. Deng and Z. Hou, The group of isometries of a Finsler
space, Pacific J. Math., 207 (2002), 149-157.

\bibitem[DMW86]{DMW86} I. Dotti-Miatello, R. Miatello and J. A. Wolf, Bounded
isometries and homogeneous Riemannian quotient manifolds, Geom. Dedicata, 21
(1986), 21-27.

\bibitem[DX02]{DX1} S. Deng and M. Xu, Clifford-Wolf translations of Finsler spaces,
Forum Math., doi:10.1015/forum-2012-0032.

\bibitem[DX03-1]{DX2} S. Deng and M. Xu, Clifford-Wolf translations of homogeneous
Randers spheres, Israel J. Math., doi:10.1007/s11856-013-0037-4.

\bibitem[DX03-2]{DX3} S. Deng and M. Xu, Clifford-Wolf translations of left
invariant Randers metrics on compact Lie groups, Quart. J. Math.,
doi:10.1093/qmath/hat003.

\bibitem[DR83]{DR83} M. J. Druetta, Clifford translations in manifolds without
focal points, Geom. Dedicata, 14 (1983), 95-103.

\bibitem[FR63]{FR63} H. Freudenthal, Clifford-Wolf-isometrien symmetrischer ra\"{u}me,
Math. Ann., 150 (1963), 136-149.

\bibitem[HE74]{HE74} E. Heintze, On homogeneous manifolds of negative curvature, Math.
Ann., 211 (1974), 23-34.

\bibitem[OT76]{OT76} T. Ochiai and Takahashi, The group of isometries of a left
invariant metric on a Lie group, Math. Ann., 223 (1976), 91-96.

\bibitem[OZ69]{OZ69} V. Ozols, Critical points of the displacement function of an
isometry, J. Differential Geometry, 3 (1969), 411-432.

\bibitem[OZ74]{OZ74} V. Ozols, Clifford translations of symmetric spaces, Proc. Amer.
Math. Soc., 44 (1974), 169-175.

\bibitem[OZ77]{OZ77} H. Ozeki, On a transitive transformation group
of a compact group manifold,  Osaka J. Math., 14 (1977) 519-531.

\bibitem[RA41]{RA41} G. Randers, On an asymmetrical metric in the four-space of
general relativity, Phys. Rev., 59 (1941), 195-199.

\bibitem[SH01]{SH01} Z. Shen, Differential geometry of sprays and Finsler spaces,
Kluwer, Dordrent, 2001.

\bibitem[SH02]{SH02} Z. Shen, Finsler spaces with $K=0$ and $S=0$, Canadian J. Math.,
55 (2003), 112-132.

\bibitem[WH65]{WH65} H. Whitney, Local properties of analytic varieties,
Differential and Combinatorial Topology, Princeton Univ. Press, (1965), 205-244.

\bibitem[WO60]{WO60} J. A. Wolf, Sur la classification des varietes riemanniannes
homogenes a courbure constante, C. R. Math. Acad. Sci. Paris, 250 (1960),
3443-3445.

\bibitem[WO62]{WO62} J. A. Wolf, Locally symmetric homogeneous spaces, Commentarii
Mathematici Helvetici, 37 (1962/63), 65-101.

\bibitem[WO64]{WO64} J. A. Wolf, Homogeneity and bounded isometrics in manifolds
of negative curvature, Illinois J. Math., 8 (1964), 14-18.

\bibitem[WO10]{WO10} J. A. Wolf, Spaces of constant curvature, 6th ed., Surveys
and monographs, Amer. Math. Soc., 2010.

\bibitem[XD03]{XD1} M. Xu and S. Deng, Clifford-Wolf homogeneous Randers spaces,
J. of Lie Theory, 23 (2013) 837-845.
\end{thebibliography}
\end{document}